\documentclass{amsart}
\usepackage[breaklinks,colorlinks,linkcolor=blue,citecolor=blue,urlcolor=blue]{hyperref}
\usepackage{amssymb}
\usepackage{amsmath}
\usepackage{graphicx}
\graphicspath{ {images/} }
\newtheorem{theorem}{Theorem}[section]
\newtheorem{lemma}[theorem]{Lemma}
\newtheorem{corollary}[theorem]{Corollary}

\newtheorem{definition}[theorem]{Definition}

\newtheorem{question}[theorem]{Question}
\numberwithin{equation}{section}
\numberwithin{equation}{section}

\begin{document}

\title[Some characterizations of ideal variants of Hurewicz type covering...]{Some characterizations of ideal variants of Hurewicz type covering properties}

\author{Manoj Bhardwaj$^{1}$}

\address{$^{1}$Department of Mathematics, University of Delhi, New Delhi-110007, India}
\email{manojmnj27@gmail.com}
\thanks{The first author acknowledges the fellowship grant of University Grant Commission, India.}

\author{B. K. Tyagi$^{2}$}
\address{$^{2}$Department of Mathematics, Atmaram Sanatan Dharma College, University of Delhi, New Delhi-110021, India}
\email{brijkishore.tyagi@gmail.com}

\subjclass[2010]{54D20, 54B20}


\dedicatory{}


\keywords{Hurewicz space, $\mathcal{I}$-Hurewicz property, strongly star-$\mathcal{I}$-Hurewicz,  star-$\mathcal{I}$-Hurewicz property}

\begin{abstract}
In this paper, we continue to investigate topological properties of $\mathcal{I}H$ and its two star versions namely $SS \mathcal{I} H$ and $S \mathcal{I} H$. We characterized $\mathcal{I}$-Hurewicz property by $\mathcal{I}$-Hurewicz Basis property and $\mathcal{I}$-Hurewicz measure zero property for metrizable spaces. We also characterized $\mathcal{I}$-Hurewicz property, star-$\mathcal{I}$-Hurewicz property and strongly star-$\mathcal{I}$-Hurewicz property using selection principles.
\end{abstract}

\maketitle

\section{Introduction}\label{sec1}

The study of topological properties via various changes is not a new idea in topological spaces. The study of selection principles in topology and their relations to game theory and Ramsey theory was started by Scheepers \cite{H1} (see also \cite{H2}). In the last two decades it has gained the enough importance to become one of the most active areas of set theoretic topology. So the study of covering properties (more precisely, of
selection principles) became one of the more active and prevailing areas for research in General Topology. In covering properties, Hurewicz property is one of the most important properties. A number of the results in the literature show that many topological properties can be described and characterized in terms of star covering properties (see \cite{H21,H22,H23,H24}). The method of stars has been used to study the problem of metrization of topological spaces, and for definitions of several important classical topological notions. We here study such a method in investigation of selection principles for topological spaces. 

In 1925, Hurewicz \cite{H4} (see also \cite{H5}) introduced Hurewicz property in topological spaces and studied it. This property is stronger than Lindel$\ddot{o}$f and weaker than $\sigma$-compactness. In 2001, Ko$\check{c}$inac \cite{H7}(see also \cite{H8}) introduced weakly Hurewicz property as a generalization of Hurewicz spaces. Every Hurewicz space is weakly Hurewicz. In 2004, the authors Bonanzinga, Cammaroto, Ko$\check{c}$inac \cite{H13} introduced the star version of Hurewicz property and also introduced relativization of strongly star-Hurewicz property. Continuing this, in 2013, the authors Song and Li \cite{H9} introduced and studied almost Hurewicz property in topological spaces. In 2016, Ko$\check{c}$inac \cite{H6} introduced and studied the notion of mildly Hurewicz property. In 2018, Das et al. \cite{H16} introduced strongly star-$\mathcal{I}$-Hurewicz and star-$\mathcal{I}$-Hurewicz properties in topological spaces. Here we studied the strongly star-$\mathcal{I}$-Hurewicz and star-$\mathcal{I}$-Hurewicz properties and provided some examples to make a complete study of these properties.

This paper is organized as follows. In section-2, the definitions of the terms used in this paper are provided. In section-3, $\mathcal{I}$-Hurewicz property is characterized using $\mathcal{I}$-Hurewicz basis property and $\mathcal{I}$-Hurewicz measure zero property. In section-4, the concept of strongly star-$\mathcal{I}$-Hurewicz property is characterized. In section-5, the star-$\mathcal{I}$-Hurewicz property is studied.

\section{Preliminaries}\label{sec2}

Let $(X,\tau)$ or $X$ be a topological space. We will denote by $Cl(A)$ and $Int(A)$ the closure of $A$ and the interior of $A$, for a subset $A$ of $X$, respectively. The cardinality of a set $A$ is denoted by $|A|$. Let $\omega$ be the first infinite cardinal and $\omega_1$ the first uncountable cardinal. As usual, a cardinal is the initial ordinal and an ordinal is the set of smaller ordinals. Every cardinal is often viewed as a space with the usual order topology. For the terms and symbols that we do not define follow \cite{H10}. The basic definitions are given.

A nonempty collection $\mathcal{I}$ of subsets of $X$ is called an ideal in $X$ if it has the following properties:
(i) If $A \in \mathcal{I}$ and $B \subseteq A$, then $B \in \mathcal{I}$ (hereditary)
(ii) If $A \in \mathcal{I}$ and $B \in \mathcal{I}$, then $A \cup B \in \mathcal{I}$ ( finite additivity).
According to Rose and Hamlet \cite{DT} $(X,\tau,I)$ denotes a set $X$ with a topology $\tau$ and an ideal $\mathcal{I}$ on $X$. For a subset $A \subseteq X, A^\star(\mathcal{I}) = \{x \in X : U \cap A \notin \mathcal{I},$ for all open sets $U$ containing $x \}$ is called the local function of $A$ with respect to $\mathcal{I}$ and $\tau$ \cite{J}. In 1999 J. Dontchev et a1.\cite{J} called a subset $A$ of a space $(X,\tau,\mathcal{I})$ to be $\mathcal{I}$-dense if every point of $X$ is in the local function of $A$ with respect $\mathcal{I}$ and $\tau$, that is, if $A^\star(\mathcal{I}) = X$. An ideal $\mathcal{I}$ is codense if $\mathcal{I} \cap \tau = \{\emptyset\}$. We denote the ideal of all nowhere dense sets by $\mathcal{I}_{ND}$ and ideal of all finite(countable) sets by $\mathcal{I}_{fin}(\mathcal{I}_{c})$. An ideal is $\sigma$-ideal if it is closed under countable additivity, that is if $A_n \in \mathcal{I}$, then $\cup A_n \in \mathcal{I}$. For more details of ideals follow \cite{JH} and \cite{K}.

Throughout the paper $\mathcal{I}$ denotes a proper admissible ideal of subsets of natural numbers.

Here, as usual, for a subset $A$ of a space $X$ and a collection $\mathcal{P}$ of subsets of
$X$, $St(A, \mathcal{P})$ denotes the star of $A$ with respect to $\mathcal{P}$, that is the set $\bigcup \{ P \in \mathcal{P} : A \cap P \neq \emptyset\}$; for $A = \{x\}$, $x \in X$, we write $St(x,\mathcal{P})$ instead of $St(\{x\}, \mathcal{P})$.

Let $\mathcal{A}$ and $\mathcal{B}$ be collections of open covers of a topological space $X$.

The symbol $S_1(\mathcal{A}, \mathcal{B})$ denotes the selection hypothesis that for each sequence $<\mathcal{U}_n : n \in \omega>$ of elements of $\mathcal{A}$ there exists a sequence $<U_n : n \in \omega>$ such that for each $n$, $U_n \in \mathcal{U}_n$ and $\{U_n : n \in \omega\} \in \mathcal{B}$ \cite{H1}.

The symbol $S_{fin}(\mathcal{A}, \mathcal{B})$ denotes the selection hypothesis that for each sequence $<\mathcal{U}_n : n \in \omega>$ of elements of $\mathcal{A}$ there exists a sequence $<\mathcal{V}_n : n \in \omega>$ such that for each $n$, $\mathcal{V}_n$ is a finite subset of $\mathcal{U}_n$ and $\bigcup_{n \in \omega} \mathcal{V}_n$ is an element of $\mathcal{B}$ \cite{H1}.

In \cite{H25}, Ko$\check{c}$inac introduced star selection principles in the following way.

The symbol $S^\star_1(\mathcal{A}, \mathcal{B})$ denotes the selection hypothesis that for each sequence $<\mathcal{U}_n : n \in \omega>$ of elements of $\mathcal{A}$ there exists a sequence $<U_n : n \in \omega>$ such that for each $n$, $U_n \in \mathcal{U}_n$ and $\{St(U_n, \mathcal{U}_n) : n \in \omega\} \in \mathcal{B}$.

The symbol $S^\star_{fin}(\mathcal{A}, \mathcal{B})$ denotes the selection hypothesis that for each sequence $<\mathcal{U}_n : n \in \omega>$ of elements of $\mathcal{A}$ there exists a sequence $<\mathcal{V}_n : n \in \omega>$ such that for each $n$, $\mathcal{V}_n$ is a finite subset of $\mathcal{U}_n$ and $\bigcup_{n \in \omega} \{St(V, \mathcal{U}_n) : V \in \mathcal{V}_n\}$ is an element of $\mathcal{B}$

The symbol $U^\star_{fin}(\mathcal{A}, \mathcal{B})$ denotes the selection hypothesis that for each sequence $<\mathcal{U}_n : n \in \omega>$ of elements of $\mathcal{A}$ there exists a sequence $<\mathcal{V}_n : n \in \omega>$ such that for each $n$, $\mathcal{V}_n$ is a finite subset of $\mathcal{U}_n$ and $\{St(\bigcup \mathcal{V}_n, \mathcal{U}_n) : n \in \omega\} \in \mathcal{B}$ or there is some $n$ such that $St(\bigcup \mathcal{V}_n, \mathcal{U}_n) = X$.

Let $\mathcal{K}$ be a family of subsets of $X$. Then we say that $X$ belongs to the class
$SS^\star_\mathcal{K}(\mathcal{A}, \mathcal{B})$ if $X$ satisfies the following selection hypothesis that for every sequence $<\mathcal{U}_n : n \in \omega>$ of elements of $\mathcal{A}$ there exists a sequence $<K_n : n \in \omega>$ of elements of $K$ such that $\{St(K_n,\mathcal{U}_n) : n \in \omega\} \in \mathcal{B}$.

When $\mathcal{K}$ is the collection of all one-point [resp., finite, compact] subspaces of
$X$ we write $SS^\star_1(\mathcal{A}, \mathcal{B})$ [resp., $SS^\star_{fin}(\mathcal{A}, \mathcal{B})$, $SS^\star_{comp}(\mathcal{A}, \mathcal{B})$] instead of $SS^\star_\mathcal{K}(\mathcal{A}, \mathcal{B})$.

In this paper $\mathcal{A}$ and $\mathcal{B}$ will be collections of the following open covers of a space $X$:

$\mathcal{O}$ : the collection of all open covers of $X$.

$\Omega$ : the collection of $\omega$-covers of $X$. An open cover $\mathcal{U}$ of $X$ is an $\omega$-cover \cite{H27} if $X$ does not belong to $\mathcal{U}$ and every finite subset of $X$ is contained in an element of $\mathcal{U}$.

$\Gamma$ : the collection of $\gamma$-covers of $X$. An open cover $\mathcal{U}$ of $X$ is a $\gamma$-cover \cite{H27} if it is infinite and each $x \in X$ belongs to all but finitely many elements of $\mathcal{U}$.

$\mathcal{O}^{gp}$ : the collection of groupable open covers. An open cover $\mathcal{U}$ of $X$ is groupable \cite{H28} if it can be expressed as a countable union of finite, pairwise
disjoint subfamilies $\mathcal{U}_n$, such that each $x \in X$ belongs to $\bigcup \mathcal{U}_n$ for all but finitely many $n$.

$\mathcal{I}-\mathcal{O}^{gp}$ denotes the collection of all $\mathcal{I}$-groupable open covers of $X$; a open cover $\mathcal{U}$ of $X$ is $\mathcal{I}$-groupable \cite{H39} if it can be represented in the form $\mathcal{U} = \bigcup_{n \in \omega} \mathcal{U}_n$, where $\mathcal{U}_n$'s are finite, pairwise disjoint and each $x \in X$, $\{n \in \omega : x \notin \bigcup \mathcal{U}_n\} \in \mathcal{I}$.

A nonempty collection $\mathcal{I}$ of subsets of $X$ is called an ideal in $X$ if it has the following properties:
(i) If $A \in \mathcal{I}$ and $B \subseteq A$, then $B \in \mathcal{I}$ (hereditary)
(ii) If $A \in \mathcal{I}$ and $B \in \mathcal{I}$, then $A \cup B \in \mathcal{I}$ ( finite additivity).
According to Rose and Hamlet \cite{DT} $(X,\tau,I)$ denotes a set $X$ with a topology $\tau$ and an ideal $\mathcal{I}$ on $X$. For a subset $A \subseteq X, A^\star(\mathcal{I}) = \{x \in X : U \cap A \notin \mathcal{I},$ for all open sets $U$ containing $x \}$ is called the local function of $A$ with respect to $\mathcal{I}$ and $\tau$ \cite{J}. In 1999 J. Dontchev et a1.\cite{J} called a subset $A$ of a space $(X,\tau,\mathcal{I})$ to be $\mathcal{I}$-dense if every point of $X$ is in the local function of $A$ with respect $\mathcal{I}$ and $\tau$, that is, if $A^\star(\mathcal{I}) = X$. An ideal $\mathcal{I}$ is codense if $\mathcal{I} \cap \tau = \{\emptyset\}$. We denote the ideal of all nowhere dense sets by $\mathcal{I}_{ND}$ and ideal of all finite(countable) sets by $\mathcal{I}_{fin}(\mathcal{I}_{c})$. An ideal is $\sigma$-ideal if it is closed under countable additivity, that is if $A_n \in \mathcal{I}$, then $\cup A_n \in \mathcal{I}$. For more details of ideals follow \cite{JH} and \cite{K}.

Throughout the paper $\mathcal{I}$ denotes the proper admissible ideal of subsets of natural numbers.

\begin{definition} \label{2.1} \cite{H4}
A space $X$ is said to have \textit{Hurewicz property} (in short $H$) if for each sequence $<\mathcal{U}_n : n \in \omega>$ of open covers of $X$ there is a sequence $<\mathcal{V}_n : n \in \omega>$ such that for each $n$, $\mathcal{V}_n$ is a finite subset of $\mathcal{U}_n$ and each $x \in X$ belongs to $\bigcup \mathcal{V}_n$ for all but finitely many $n$.
\end{definition}

\begin{definition} \label{2.6} \cite{H13}
A space $X$ is said to have \textit{star-Hurewicz property} (in short $SH$) if for each sequence $<\mathcal{U}_n : n \in \omega>$ of open covers of $X$ there is a sequence $<\mathcal{V}_n : n \in \omega>$ such that for each $n$, $\mathcal{V}_n$ is a finite subset of $\mathcal{U}_n$ and each $x \in X$ belongs to $St(\bigcup \mathcal{V}_n, \mathcal{U}_n)$ for all but finitely many $n$.
\end{definition}

\begin{definition} \label{2.7} \cite{H13}
A space $X$ is said to have \textit{strongly star-Hurewicz property} (in short $SSH$) if for each sequence $<\mathcal{U}_n : n \in \omega>$ of open covers of $X$ there is a sequence $<A_n : n \in \omega>$ of finite subsets of $X$ such that each $x \in X$ belongs to $St(A_n, \mathcal{U}_n)$ for all but finitely many $n$.
\end{definition}

Let $\mathcal{I}$ be an admissible ideal in $\omega$. Then we have the following two definitions.

\begin{definition} \label{} \cite{H40}
A space $X$ is said to have \textit{$\mathcal{I}$-Hurewicz property} (in short $\mathcal{I}H$) if for each sequence $<\mathcal{U}_n : n \in \omega>$ of open covers of $X$ there is a sequence $<\mathcal{V}_n : n \in \omega>$ such that for each $n$, $\mathcal{V}_n$ is a finite subset of $\mathcal{U}_n$ and for each $x \in X$, $\{n \in \omega : x \notin \bigcup \mathcal{V}_n\} \in \mathcal{I}$.
\end{definition}

\begin{definition} \label{} \cite{H16}
A space $X$ is said to have \textit{star-$\mathcal{I}$-Hurewicz property} (in short $S\mathcal{I}H$) if for each sequence $<\mathcal{U}_n : n \in \omega>$ of open covers of $X$ there is a sequence $<\mathcal{V}_n : n \in \omega>$ such that for each $n$, $\mathcal{V}_n$ is a finite subset of $\mathcal{U}_n$ and for each $x \in X$, $\{n \in \omega : x \notin St(\bigcup \mathcal{V}_n, \mathcal{U}_n)\} \in \mathcal{I}$.
\end{definition}

\begin{definition} \label{} \cite{H16}
A space $X$ is said to have \textit{strongly star-$\mathcal{I}$-Hurewicz property} (in short $SS\mathcal{I}H$) if for each sequence $<\mathcal{U}_n : n \in \omega>$ of open covers of $X$ there is a sequence $<V_n : n \in \omega>$ of finite subsets of $X$ and for each $x \in X$, $\{n \in \omega : x \notin St(V_n, \mathcal{U}_n)\} \in \mathcal{I}$.
\end{definition}

\begin{definition}   \cite{H24,H22}
A topological space $X$ is strongly starcompact if for every open cover $\mathcal{U}$ of $X$ there exists a finite subset $A$ of $X$ such that $X = St(A, \mathcal{U})$
\end{definition}
Call a space $\sigma$-strongly starcompact if it is a union of countably many strongly starcompact spaces.

\begin{definition} \label{2.5} \cite{H10}
A function $f$ from a topological space $X$ to a space $Y$ is said to be perfect map if 
\begin{enumerate}
\item
$f$ is onto
\item
$f$ is continuous
\item
$f$ is closed map
\item
$f^{-1}(y)$ is compact in $X$ for each $y \in Y$.
\end{enumerate}
\end{definition}

\section{Properties of $\mathcal{I}H$ spaces}

The symbol $\mathcal{I}H(X)$ denotes the following $\mathcal{I}$-Hurewicz game \cite{H39} on $X$:
players ONE and TWO play a round for each $n \in \omega$. In the $n$th round
player ONE chooses a open cover $\mathcal{U}_n$ for $X$ and then TWO chooses a finite
set $\mathcal{V}_n \subseteq \mathcal{U}_n$. TWO wins a play $\mathcal{U}_1; \mathcal{V}_1; \mathcal{U}_2; \mathcal{V}_2...$ if each $x \in X$, $\{n \in \omega : x \notin \bigcup \mathcal{V}_n\} \in \mathcal{I}$; otherwise ONE wins.

\begin{lemma} \label{a} \cite{H16}
An $\mathcal{I}H$ space is Lindel$\ddot{o}$f.
\end{lemma}

\textbf{Characterizations of $\mathcal{I}$-Hurewicz property}

From Lemma \ref{a}, it is clear that spaces $X$ having the $\mathcal{I}$-Hurewicz property satisfy: each open cover of $X$ has a countable subcover, that is, each $\mathcal{I}$-Hurewicz space is Lindel$\ddot{o}$f. Therefore, when we work with the $\mathcal{I}$-Hurewicz property, we may assume that all open covers of a space are countable. Note also that a space $X$ has the $\mathcal{I}$-Hurewicz property whenever ONE does not have the winning strategy in the game $\mathcal{I}H(X)$.

Notice that any $\mathcal{U} \in \Omega$ satisfies:
For each $k$ and each partition $\mathcal{U} = \mathcal{U}_1 \cup ... \cup \mathcal{U}_k$ there is an $i \leq k$ with $\mathcal{U}_i \in \Omega$.

A space $X$ is called $\omega$-Lindel$\ddot{o}$f if each cover in $\Omega$ has a countable subcover.

In next theorem, we talk about countable covers.

\begin{theorem}
For a Lindel$\ddot{o}$f space $X$ the following statements are equivalent:
\begin{enumerate}
\item
X satisfies $S_{fin}(\mathcal{O},\mathcal{O})$ and $\mathcal{O} = \mathcal{I}-\mathcal{O}^{gp}$;
\item
X satisfies $S_{fin}(\mathcal{O}, \mathcal{I}-\mathcal{O}^{gp})$;
\item
ONE has no winning strategy in the game $G_{fin}(\mathcal{O}, \mathcal{I}-\mathcal{O}^{gp})$;
\end{enumerate}
\end{theorem}
\begin{proof}
(1) $\Leftrightarrow$ (2).
The proof is easy and thus omitted.

(1) $\Rightarrow$ (3).
Let $\sigma$ be a strategy for ONE in the game $G_{fin}(\mathcal{O}, \mathcal{I}-\mathcal{O}^{gp})$. Then it is also a strategy for ONE in $G_{fin}(\mathcal{O}, \mathcal{O})$. By the given condition $X$ has the property $S_{fin}(\mathcal{O},\mathcal{O})$ and so by [Theorem 10,\cite{H4}] this is not a winning strategy for ONE in $G_{fin}(\mathcal{O}, \mathcal{O})$. For a $\sigma$-play $\sigma(\emptyset) = \mathcal{U}_1, \mathcal{V}_1; \sigma(\mathcal{V}_1) = \mathcal{U}_2, \mathcal{V}_2;$ ... lost by ONE, TWO’s moves constitute an open cover of $X$. Again by the given condition $\mathcal{O} = \mathcal{I}-\mathcal{O}^{gp}$, this open cover is $\mathcal{I}$-groupable. Hence this play is actually lost by ONE in the game $G_{fin}(\mathcal{O}, \mathcal{I}-\mathcal{O}^{gp})$.

(3) $\Rightarrow$ (1).
Let $<\mathcal{U}_n : n \in \omega>$ be a sequence of open covers of $X$. Since ONE has no winning strategy in the game $G_{fin}(\mathcal{O}, \mathcal{I}-\mathcal{O}^{gp})$,
there is a $\sigma$-play $\sigma(\emptyset) = \mathcal{U}_1, \mathcal{V}_1; \sigma(\mathcal{V}_1) = \mathcal{U}_2, \mathcal{V}_2;$ ... lost by ONE, TWO’s moves constitute a $\mathcal{I}$-groupable open cover of $X$. Then $S_{fin}(\mathcal{O},\mathcal{O})$ holds. Observe that $\mathcal{O} \supseteq \mathcal{I}-\mathcal{O}^{gp}$. For the rest, let $\mathcal{U}$ be any open cover of $X$. Now define $\mathcal{U}_n = \mathcal{U}$ for each $n$. Then $<\mathcal{U}_n : n \in \omega>$ be a sequence of open covers of $X$. Since ONE has no winning strategy in the game $G_{fin}(\mathcal{O}, \mathcal{I}-\mathcal{O}^{gp})$, there is a $\sigma$-play $\sigma(\emptyset) = \mathcal{U}_1, \mathcal{V}_1; \sigma(\mathcal{V}_1) = \mathcal{U}_2, \mathcal{V}_2;$ ... lost by ONE, TWO’s moves constitute a $\mathcal{I}$-groupable open cover of $X$. Hence each open cover contains an $\mathcal{I}$-groupable open subcover.
\end{proof}

Now we try to characterize $\mathcal{I}H$ property using ideals in certain types of open covers. For it, we need the following definition.

\begin{definition}  \cite{H1}
Let $\mathcal{A}$ and $\mathcal{B}$ be families of subsets of the infinite set $S$. Then $CDR_{sub}(\mathcal{A},\mathcal{B})$ denotes the statement that for each sequence $<A_n : n \in \omega>$ of elements of $\mathcal{A}$ there is a sequence $<B_n : n \in \omega>$ such that for each $n$, $B_n \subseteq A_n$, for $m \neq n$, $B_m \cap B_n = \emptyset$, and each $B_n$ is a member of $\mathcal{B}$.
\end{definition}

\begin{theorem}
If a space $X$ has $\mathcal{I}H$ property and $CDR_{sub}(\mathcal{O},\mathcal{O})$ holds for $X$, then $S_{fin}(\mathcal{O}, \mathcal{I}-\mathcal{O}^{gp})$ holds.
\end{theorem}
\begin{proof}
Let $<\mathcal{U}_n : n \in \omega>$ be a sequence of open covers of $X$. Since $X$ has the property $CDR_{sub}(\mathcal{O},\mathcal{O})$ we may assume that $\mathcal{U}_n$'s are pairwise disjoint. Since $X$ has $\mathcal{I}H$ property there is a sequence $<\mathcal{V}_n : n \in \omega>$ such that for each $n$, $\mathcal{V}_n$ is a finite subset of $\mathcal{U}_n$ and for each $x \in X$, $\{n \in \omega : x \notin \bigcup \mathcal{V}_n\} \in \mathcal{I}$. Then $\mathcal{V}_n$'s are pairwise disjoint and hence $\bigcup_{n \in \omega} \mathcal{V}_n$ is an $\mathcal{I}$-groupable cover of $X$.
\end{proof}

If we use stronger condition than $\mathcal{I}H$ property, then we can drop $CDR_{sub}(\mathcal{O},\mathcal{O})$.

\begin{theorem} 
If One does not have a winning strategy in $\mathcal{I}H(X)$ game, then $S_{fin}(\Lambda, \mathcal{I}-\mathcal{O}^{gp})$ holds.
\end{theorem}
\begin{proof}
Let $X$ be a topological space such that One does not have a winning strategy in $\mathcal{I}H(X)$ game. Let $<\mathcal{U}_n : n \in \omega>$ be a sequence of large covers of $X$. 
Consider the following strategy $\sigma$ of ONE in the $\mathcal{I}$-Hurewicz game on $X$. In the first inning ONE plays $\sigma(\emptyset) = \mathcal{U}_1$. If TWO responds with the finite set $\mathcal{T}_1$, then ONE plays $\sigma(\mathcal{T}_1) = \mathcal{U}_1 \setminus \mathcal{T}_1$. If TWO responds with the finite set $\mathcal{T}_2 \subseteq \sigma(\mathcal{T}_1)$, then ONE plays $\sigma(\mathcal{T}_1, \mathcal{T}_2) = \mathcal{U}_3 \setminus (\mathcal{T}_1 \bigcup \mathcal{T}_2)$ and so on. Since $\sigma$ is not a winning strategy for ONE, consider a $\sigma$ play $\sigma(\emptyset), \mathcal{T}_1, \sigma(\mathcal{T}_1), \mathcal{T}_2, \sigma(\mathcal{T}_1, \mathcal{T}_2), \mathcal{T}_3, ...$ which is lost by ONE. Then the sequence $<\mathcal{T}_n : n \in \omega>$ of moves by TWO are disjoint from each other (by the definition of the strategy $\sigma$) and for each $x \in X$ the set $\{n \in \omega : x \notin \bigcup \mathcal{T}_n\} \in \mathcal{I}$.
\end{proof}

It can be noted that for a space $X$, $S_{fin}(\mathcal{O}, \mathcal{I}-\mathcal{O}^{gp}) \Rightarrow S_{fin}(\Lambda, \mathcal{I}-\mathcal{O}^{gp}) \Rightarrow S_{fin}(\Omega, \mathcal{I}-\mathcal{O}^{gp})$, since each large cover is an open cover and each $\omega$-cover is a large cover and hence $\mathcal{I}$-Hurewicz property with $CDR_{sub}(\mathcal{O},\mathcal{O})$ implies $S_{fin}(\Omega, \mathcal{I}-\mathcal{O}^{gp})$. 

For the converse we consider inverse invariant ideal. An ideal $\mathcal{I}$ of $\omega$ is called inverse invariant \cite{H41} if for an increasing function $f : \omega \rightarrow \omega$ such that $f(n) \geq n$ for all $n \in \omega$, $f(A) \in \mathcal{F}(\mathcal{I})$ implies that $A \in \mathcal{F}(\mathcal{I})$.

\begin{theorem} \label{35}
Let $\mathcal{I}$ be an inverse invariant ideal and $X$ be an $\omega$-Lindel$\ddot{o}$f space. If $X$ satisfies $S_{fin}(\Omega, \mathcal{I}-\mathcal{O}^{gp})$, then $X$ has the $\mathcal{I}$-Hurewicz property provided that $S_1(\mathcal{F}(\mathcal{I}), \mathcal{F}(\mathcal{I}))$ holds.
\end{theorem}
\begin{proof}
Let $<\mathcal{U}_n : n \in \omega>$ be a sequence of countable open covers of $X$. Let $M$
be the set of all those positive integers $n$ such that $\mathcal{U}_n$ contains a finite subset covering $X$. If $M \in \mathcal{F}(\mathcal{I})$ then there is nothing to prove. Otherwise rewriting $<\mathcal{U}_n : n \in \omega \setminus M>$ as $<\mathcal{U}_n : n \in \omega>$ we have to show that each of these $\mathcal{U}_n$ has a finite sub-family satisfying the condition of $\mathcal{I}$-Hurewicz property. For each $n$ define the set $\mathcal{V}_n$ to be the set of finite unions of elements of $\mathcal{U}_n$. Then each $\mathcal{V}_n$ is in $\Omega$. Since $X$ is $\omega$-Lindel$\ddot{o}$f, there is a countable $\omega$-subcover $\mathcal{W}_n$, say $\mathcal{W}_n = \{V_{n,k} : k \in \omega\}$. From covers $\mathcal{W}_n$, $n \in \omega$, we define new covers $\mathcal{P}_n \in \Omega$ in the following way :
\begin{center}
$n =1$ : $\mathcal{P}_1 = \mathcal{W}_1$; \\
$n > 1$ : $\mathcal{P}_n = \{V_{1,m_1} \cap V_{2,m_2} \cap ... \cap V_{n,m_n} : n < m_1 < m_2 <...<m_n \} \setminus \{\emptyset\}$. 
\end{center}
For each element of $\mathcal{P}_n$ choose a representation of the form $V_{1,m_1} \cap V_{2,m_2} \cap ... \cap V_{n,m_n}$ with $n < m_1 < m_2 <...<m_n$.
Apply (2) to the sequence $<\mathcal{P}_n : n \in \omega>$ to find for each $n$ a finite set $\mathcal{A}_n \subset \mathcal{P}_n$ such that $\bigcup_{n \in \omega} \mathcal{A}_n \in \mathcal{I}-\mathcal{O}^{gp}$. Then we can choose finite, pairwise disjoint sets $\mathcal{H}_n$, $n \in \omega$, such that $\bigcup_{n \in \omega} \mathcal{A}_n = \bigcup_{n \in \omega} \mathcal{H}_n$, and each $x \in X$, $\{n \in \omega : x \notin \bigcup \mathcal{H}_n\} \in \mathcal{I}$.

Now for $n$, define $A_n = \{k \in \omega : \mathcal{H}_k \subseteq \bigcup_{i > n} \mathcal{A}_i\}$. Since each $\mathcal{H}_n$ is finite and pairwise disjoint, each $A_n$ is cofinite subset of natural numbers and hence belongs to $\mathcal{F}(\mathcal{I})$(as $\mathcal{I}$ is admissible). As $S_1(\mathcal{F}(\mathcal{I}), \mathcal{F}(\mathcal{I}))$ holds, there is a sequence $<k_i :  i \in \omega>$ such that $k_i \in A_i$ for each $i$, and $\{k_1 < k_2 < ... \} \in \mathcal{F}(\mathcal{I})$. Then $\mathcal{H}_{k_1} \subseteq \bigcup_{i > 1} \mathcal{A}_i$. Let $\mathcal{G}_1$ denote the set of all $V_{1,p}$, that occurs as a first components in the chosen representations of elements of $\mathcal{H}_{k_1}$. Now $\mathcal{H}_{k_2} \subseteq \bigcup_{i > 2} \mathcal{A}_i$. Let $\mathcal{G}_2$ denote the set of all $V_{1,p}$, that occurs as a second components in the chosen representations of elements of $\mathcal{H}_{k_2}$. Continuing in this way, we obtain finite sets $\mathcal{G}_n \subset \mathcal{V}_n$.

Since $\mathcal{I}$ is an inverse invariant ideal, for $f : \omega \rightarrow \omega$ such that $f(i) = k_i$, $f(A) \in \mathcal{F}(\mathcal{I})$ implies that $A \in \mathcal{F}(\mathcal{I})$. For each $x \in X$, $\{n \in \omega : x \in \bigcup \mathcal{G}_n\} \supseteq \{n \in \omega : x \in \bigcup \mathcal{H}_{k_n}\}$ and $f(\{n \in \omega : x \in \bigcup \mathcal{H}_{k_n}\}) = \{k_n \in \omega : x \in \bigcup \mathcal{H}_{k_n}\} = \{n \in \omega : x \in \bigcup \mathcal{H}_n\} \cap \{k_n : n \in \omega\} \in \mathcal{F}(\mathcal{I})$, since $\{n \in \omega : x \in \bigcup \mathcal{H}_n\}, \{k_n : n \in \omega\} \in \mathcal{F}(\mathcal{I})$. So $\{n \in \omega : x \in \bigcup \mathcal{G}_n\}  \in \mathcal{F}(\mathcal{I})$. 
For each element $F$ of $\mathcal{G}_n$ choose finitely many elements of $\mathcal{U}_n$ whose union is $F$ and let $\mathcal{L}_n$ denote the finite set of elements of $\mathcal{U}_n$ chosen in this way and $\bigcup \mathcal{G}_n = \bigcup \mathcal{L}_n$ for each $n$. Then the sequence $<\mathcal{L}_n : n \in \omega>$ witnesses for $<\mathcal{U}_n : n \in \omega>$ the $\mathcal{I}$-Hurewicz property of $X$.
\end{proof}

\begin{corollary} \label{30}
Let $\mathcal{I}$ be an inverse invariant ideal and $S_1(\mathcal{F}(\mathcal{I}), \mathcal{F}(\mathcal{I}))$ holds. For an $\omega$-Lindel$\ddot{o}$f space $X$ for which $CDR_{sub}(\mathcal{O},\mathcal{O})$ holds, the following statements are equivalent :
\begin{enumerate}
\item
$X$ has $\mathcal{I}H$ property;
\item
$S_{fin}(\mathcal{O}, \mathcal{I}-\mathcal{O}^{gp})$ holds;
\item 
$S_{fin}(\Lambda, \mathcal{I}-\mathcal{O}^{gp})$ holds;
\item 
$S_{fin}(\Omega, \mathcal{I}-\mathcal{O}^{gp})$ holds.
\end{enumerate}
\end{corollary}

\begin{theorem}
Let $\mathcal{I}$ be an inverse invariant ideal and $X$ be a Lindel$\ddot{o}$f space. If $X$ satisfies $S_{fin}(\mathcal{O}, \mathcal{I}-\mathcal{O}^{gp})$, then $X$ has the $\mathcal{I}$-Hurewicz property provided that $S_1(\mathcal{F}(\mathcal{I}), \mathcal{F}(\mathcal{I}))$ holds.
\end{theorem}
\begin{proof}
The proof is similar with the proof of Theorem \ref{35} after necessary modifications.
\end{proof}

\begin{corollary} \label{44}
Let $\mathcal{I}$ be an inverse invariant ideal and $S_1(\mathcal{F}(\mathcal{I}), \mathcal{F}(\mathcal{I}))$ holds. For a Lindel$\ddot{o}$f space $X$ for which $CDR_{sub}(\mathcal{O},\mathcal{O})$ holds, the following statements are equivalent :
\begin{enumerate}
\item
$X$ has $\mathcal{I}H$ property;
\item
X satisfies $S_{fin}(\mathcal{O},\mathcal{O})$ and $\mathcal{O} = \mathcal{I}-\mathcal{O}^{gp}$;
\item
X satisfies $S_{fin}(\mathcal{O}, \mathcal{I}-\mathcal{O}^{gp})$;
\item
ONE has no winning strategy in the game $G_{fin}(\mathcal{O}, \mathcal{I}-\mathcal{O}^{gp})$.

\end{enumerate}
\end{corollary}

Recall the following notion in Ramsey theory, called the Baumgartner-Taylor partition relation (see \cite{H36}). For each positive integer $k$,
\begin{center}
$\mathcal{A} \rightarrow [\mathcal{B}]^{2}_k$
\end{center}
denotes the following statement:
For each $A \in \mathcal{A}$ and for each function $f : [A]^2 \rightarrow \{1, ... , k\}$ there are a set $B \in \mathcal{B}$ with $B \subseteq A$, a $j \in \{1, ... , k\}$, and a partition $B = \bigcup_{n \in \omega} B_n$ of $B$ into pairwise disjoint finite sets such that for each $\{a,b\} \in [\mathcal{B}]^2$ for which $a$ and $b$ are not from the same $B_n$, we have $f(\{a, b\}) = j$.
Here $[\mathcal{A}]^2$ denotes the set of all two-element subsets of $A$.

\begin{theorem}
Let $X$ be a Lindel$\ddot{o}$f space satisfying the following condition :
for each $k \in \omega$ the partition relation $\mathcal{O} \rightarrow [\mathcal{I}-\mathcal{O}^{gp}]^2_k$ holds,
then $X$ has satisfies $S_{fin}(\mathcal{O}, \mathcal{I}-\mathcal{O}^{gp})$.
\end{theorem}
\begin{proof}
Let $<\mathcal{U}_n : n \in \omega>$ be a sequence of countable elements of $\mathcal{O}$. Let $\mathcal{U}_n = \{U_{n,k} : k \in \omega\}$. Define $\mathcal{V}$ to be collection of nonempty sets of the form $U_{1,n} \cap U_{n,k}$. It is clear that $\mathcal{V} \in \mathcal{O}$. For each $V \in \mathcal{V}$ choose a representation of the form $V = U_{1,n} \cap U_{n,k}$. Then define the function $f : [\mathcal{V}]^2 \rightarrow \{1,2\}$ as : 
\begin{center}
$f(\{U_{1,n_1} \cap U_{n_1,k},U_{1,n_2} \cap U_{n_2,j}\}) = 1$ if $n_1 = n_2$ and $2$ otherwise.
\end{center}
By given hypothesis, choose a nearly homogeneous of color $j$ set $\mathcal{W} \subset \mathcal{V}$ with $\mathcal{W} \in \mathcal{I}-\mathcal{O}^{gp}$. Assume $\mathcal{W} = \bigcup_{k \in \omega} \mathcal{W}_k$ is a sequence of finite, pairwise disjoint sets such that for $A$ and $B$ from distinct $\mathcal{W}_k$'s we have $f(\{A,B\}) = j$. We have the following two possibilities.

Case 1 : $j = 1$. There is an $n$ such that for all $A \in \mathcal{W}$ we have $A \subset U_{1,n}$. This implies that $\mathcal{W}$ does not belong to $\mathcal{I}-\mathcal{O}^{gp}$. Thus, this case does not hold.

Case 2 : $j = 2$. For each $n > 1$ define $\mathcal{H}_n = \{U_{n,i} : U_{n,i}$ is the second coordinate in the representation of an $W \in \mathcal{W}\}$ and set $\mathcal{H} = \bigcup_{n \in \omega} \mathcal{H}_n$. Then $\mathcal{H}$ is the union of finite subsets of $\mathcal{U}_n$, $n \in \omega$ since $\mathcal{W} = \{U_{1,1} \cap U_{1,j} : j \in \{i^1_1,i^1_2,...,i^1_{n_1}\}\} \cup \{U_{1,2} \cap U_{2,j} : j \in \{i^2_1,i^2_2,...,i^2_{n_2}\}\} \cup ...$. Then for each $k$, $\mathcal{H}_k = \{U_{1,j} : j \in \{i^k_1,i^k_2,...,i^k_{n_1}\}\}$ and hence $\mathcal{H}_n$'s are pairwise disjoint and finite. Since $\mathcal{W} \in \mathcal{I}-\mathcal{O}^{gp}$, for each $x \in X$, $\{k \in \omega : x \notin \bigcup \mathcal{H}_k\} \subseteq \{k \in \omega : x \notin \bigcup \{U_{1,1} \cap U_{1,j} : j \in \{i^k_1,i^k_2,...,i^k_{n_1}\}\}\} \in \mathcal{I}$. If for some $k$, $\mathcal{H}_k$ was not defined, we set $\mathcal{H}_k = \emptyset$. So we get the sequence $<\mathcal{H}_k : k \in \omega>$ witnessing for $<\mathcal{U}_k : k \in \omega>$ that $S_{fin}(\mathcal{O}, \mathcal{I}-\mathcal{O}^{gp})$ holds.
\end{proof}

\begin{corollary}
Let $\mathcal{I}$ be an inverse invariant ideal for which $S_1(\mathcal{F}(\mathcal{I}), \mathcal{F}(\mathcal{I}))$ holds and $X$ be a Lindel$\ddot{o}$f space. If $X$ satisfies the following condition :
for each $k \in \omega$ the partition relation $\mathcal{O} \rightarrow [\mathcal{I}-\mathcal{O}^{gp}]^2_k$ holds,
then $X$ has $\mathcal{I}$-Hurewicz property.
\end{corollary}

The converse remains as an open problem.

\textbf{$\mathcal{I}$-Hurewicz Basis property and $\mathcal{I}$-Hurewicz measure zero property}

In \cite{H47} Menger defined the following: Metric space $(X, d)$ has the Menger basis property if there is for each basis $\mathcal{B}$ of $(X, d)$ a sequence $<U_n : n \in \omega>$ such that
$lim_{n \rightarrow \infty} diam_d(U_n) = 0$, and $\{U_n : n \in \omega\}$ covers $X$. In \cite{H4} Hurewicz showed that the Menger basis property is equivalent to the Menger covering property $S_{fin}(\mathcal{O},\mathcal{O})$. When the spaces in question are metrizable the Hurewicz property can be similarly characterized by a Hurewicz basis property.
We show now that also the $\mathcal{I}$-Hurewicz property is characterized by a basis property. 

\begin{definition} \cite{H42}
A metric space $(X,d)$ is said to have Hurewicz basis property if for each basis $\mathcal{B}$ of metric space $(X,d)$ there is a sequence $<U_n : n \in \omega>$ of elements of $\mathcal{B}$ such that $\{U_n : n \in \omega\}$ is a groupable cover of $X$ and $lim_{n \rightarrow \infty} diam_d(U_n) = 0$.
\end{definition}

Now we define its ideal version.

\begin{definition} 
A metric space $(X,d)$ is said to have $\mathcal{I}$-Hurewicz basis property if for each basis $\mathcal{B}$ of metric space $(X,d)$ there is a sequence $<U_n : n \in \omega>$ of elements of $\mathcal{B}$ such that $\{U_n : n \in \omega\}$ is an $\mathcal{I}$-groupable cover of $X$ and $lim_{n \rightarrow \infty} diam_d(U_n) = 0$.
\end{definition}

For next theorem we consider inverse invariant ideals.

\begin{theorem}
Let $\mathcal{I}$ be an inverse invariant ideal such that $S_1(\mathcal{F}(\mathcal{I}),\mathcal{F}(\mathcal{I}))$ holds. If $(X,d)$ is a metric space with no isolated points for which $CDR_{sub}(\mathcal{O}, \mathcal{O})$ holds, then following statements are equivalent:
\begin{enumerate}
\item
$X$ has the $\mathcal{I}$-Hurewicz property;
\item
$X$ has the $\mathcal{I}$-Hurewicz basis property.
\end{enumerate}
\end{theorem}
\begin{proof}
Let $X$ has the $\mathcal{I}$-Hurewicz property and let $\mathcal{B}$ be a basis of $X$. Now define $\mathcal{U}_n = \{U \in \mathcal{B} : diam_d(U) < 1/(n+1) \}$. Then for each $n$, $\mathcal{U}_n$ is a large open cover of $X$. Since $X$ has the $\mathcal{I}$-Hurewicz property, by Corollary \ref{44}, there is a sequence $<\mathcal{V}_n : n \in \omega>$ such that for each $n$, $\mathcal{V}_n$ is a finite subset of $\mathcal{U}_n \subseteq \mathcal{B}$ and for each $\bigcup_{n \in \omega} \mathcal{V}_n$ is an $\mathcal{I}$-groupable cover of $X$. Then for $\bigcup_{n \in \omega} \mathcal{V}_n = \{U_n : n \in \omega\}$, $lim_{n \rightarrow \infty} diam_d(U_n) = 0$.

Conversely, let $X$ be a space having $\mathcal{I}$-Hurewicz basis property and $<\mathcal{U}_n : n \in \omega>$ be a sequence of open covers of $X$. Now assume that if an open set $V$ is a subset of an element of $\mathcal{U}_n$, then $V \in \mathcal{U}_n$. 
For each $n$ define 
\begin{center}
$\mathcal{H}_n = \{U_1 \cap U_2 \cap ...\cap U_n : (\forall i \leq n)(U_i \in \mathcal{U}_i)\} \setminus \{\emptyset\}$. 
\end{center}
Then for each $n$, $\mathcal{H}_n$ is an open cover of $X$ and has the property that if an open set $V$ is a subset of an element of $\mathcal{H}_n$, then $V \in \mathcal{H}_n$. 

Now let $\mathcal{U}$ be the set $\{U \cup V : (\exists n)(U,V \in \mathcal{H}_n$ and $diam_d(U \cup V) > 1/n)\}$. First we show that $\mathcal{U}$ is a basis for $X$. For it, let $W$ be an open subset containing a point $x$. Since $(X,d)$ does not have isolated points, $x$ is not an isolated point of $X$. Then we can choose $y \in W \setminus \{x\}$ and $n > 1$ with $d(x,y) > 1/n$. Since $\mathcal{H}_n$ is an open cover of $X$, there are $U^{'},V^{'} \in \mathcal{H}_n$ such that $x \in U^{'}$ and $y \in V^{'}$. Now put $U = U^{'} \cap W \setminus \{y\}$ and $V = V^{'} \cap W \setminus \{x\}$. Then $U,V \in \mathcal{H}_n$ since $\mathcal{H}_n$ has the property that if an open set $V$ is a subset of an element of $\mathcal{H}_n$, then $V \in \mathcal{H}_n$. Also $U \cup V \subseteq W$ and $diam_d(U \cup V) \geq d(x,y) > 1/n$. So $U \cup V \in \mathcal{U}$ and $x \in U \cup V \subseteq W$. Thus $\mathcal{U}$ is a basis for $X$.

Since $X$ has $\mathcal{I}$-Hurewicz basis property, there is a sequence $<W_n : n \in \omega>$ of elements of $\mathcal{U}$ such that $\{W_n : n \in \omega\}$ is an $\mathcal{I}$-groupable cover of $X$ and $lim_{n \rightarrow \infty} diam_d(W_n) = 0$. Then $\bigcup_{n \in \omega} \mathcal{W}_n = \{W_n : n \in \omega\}$ such that each $\mathcal{W}_n$ is finite, $\mathcal{W}_n \cap \mathcal{W}_m = \emptyset$ for $n \neq m$ and each $x \in X$, $\{n \in \omega : x \notin \bigcup \mathcal{W}_n\} \in \mathcal{I}$. Without loss of generality, let 
\begin{center}
$\mathcal{W}_1 = \{W_1,W_2,...,W_{m_1 - 1}\}$; \\
$\mathcal{W}_2 = \{W_{m_1},W_{m_1 + 1},...,W_{m_2 - 1}\}$; \\
.\\
.\\
.\\
$\mathcal{W}_n = \{W_{m_{n-1}},W_{m_{n-1} + 1},...,W_{m_n}\}$; \\
\end{center}
and so on.
Now we get a sequence $m_1 < m_2 < m_3 < ...< m_k <...$ obtained from $\mathcal{I}$-groupability of  $\{W_n : n \in \omega\}$ such that for each $x \in X$, $\{k \in \omega : x \notin \bigcup \mathcal{W}_k\} = \{k \in \omega : x \notin W_j$ for all $j$ for which $m_{k-1} \leq j < m_k\} \in \mathcal{I}$.

Since $W_n \in \mathcal{U}$, so there is $k_n$ such that $U_n,V_n \in \mathcal{H}_{k_n}$ and $W_n = U_n \cup V_n$ with $diam_d(W_n) > 1/k_n$. For each $n$, select the least $k_n$ and sets $U_n$ and $V_n$ from $\mathcal{U}_{k_n}$. Since each $\mathcal{U}_n$ has the property that if an open set $V$ is a subset of an element of $\mathcal{U}_n$, then $V \in \mathcal{U}_n$, $U_n,V_n \in \mathcal{U}_{k_n}$. Since $lim_{n \rightarrow \infty} diam_d(W_n) = 0$, for each $W_n$, there is maximal $m_n$ such that $diam_d(W_n) < 1/m_n$. Then $1/k_n < diam_d(W_n) < 1/m_n$ implies that $k_n > m_n$ for each $n$ and $lim_{n \rightarrow \infty} m_n = \infty$. Since $lim_{n \rightarrow \infty} diam_d(W_n) = 0$, so for each $k_n$, there are only finitely many $W_n$ for which the representatives $U_n,V_n$ are from $\mathcal{U}_{k_n}$ and have $diam_d(U_n \cup V_n) > 1/k_n$. Let $\mathcal{V}_{k_n}$ be the finite set of such $U_n,V_n$. 

Now choose $l_1 > 1$ so large such that each $W_i$ with $i \leq m_1$ has a representation of the form $U \cup V$ and $U$'s  and $V$'s are from the sets $\mathcal{V}_{k_i}, k_i \leq l_1$.
Then select $j_1$ so large such that for all $i > j_1$, if $W_i$ has representatives from $\mathcal{V}_{k_i}$, then $k_i > l_1$.

For choosing $l_2$, let $m_k$ be least larger then $j_1$, and now choose $l_2$ so large that if $W_i$ with $m_k \leq i <m_{k+1}$ uses a $\mathcal{V}_{k_i}$, then $k_i \leq l_2$, that is, choose maximal of $k_i$ for which $m_k \leq i <m_{k+1}$ and say $l_2$, then $l_1 < k_i \leq l_2$. Now choose maximal of $i$ for which the representation of $W_i$ from $\mathcal{V}_{k_i}$ where $l_1 < k_i \leq l_2$ and say $j_2$, then $j_2 > j_1$ and $\forall i \geq j_2$ if $W_i$ uses $\mathcal{V}_{k_n}$, then $k_n > l_2$.

Similarly alternately choose $l_m$ and $j_m$. For each $m$ if we consider the least $m_k > l_m$, then :
\begin{enumerate}
\item
if $W_i$ with $m_k \leq i <m_{k+1}$ uses a $\mathcal{V}_{k_i}$ then $l_m < k_i \leq l_{m+1}$;
\item
if $i \geq j_m$ then if $W_i$ uses $\mathcal{V}_{k_n}$ then $k_n > l_m$.
\end{enumerate}

For each $V \in \mathcal{V}_{k_n}$ with $k_n \leq l_1$, $V \in \mathcal{H}_{k_n}$ and $V \in \mathcal{U}_{k_i}$ for each $k_i \leq k_n$. Then $V \in \mathcal{U}_1$ and let $\mathcal{G}_1$ be collection of such $V \in \mathcal{V}_{k_n}$ with $k_n \leq l_1$. Then $\mathcal{G}_1 \subseteq \mathcal{U}_1$ is a finite subset.

Now for each $V \in \mathcal{V}_{k_n}$ with $l_p < k_n \leq l_{p+1}$(as $p < l_p$), $V \in \mathcal{H}_{k_n}$ and $V \in \mathcal{U}_{k_i}$ for each $k_i \leq k_n$. Then $V \in \mathcal{U}_{l_p}$ and let $\mathcal{G}_p$ be collection of such $V \in \mathcal{V}_{k_n}$ with $l_p < k_n \leq l_{p+1}$. Then $\mathcal{G}_p$ is a finite subset of $\mathcal{U}_i$. Observe that $\mathcal{G}_i \subseteq \mathcal{U}_j$ for each $j \leq i$.

For each $k$, let $A_k = \{n \in \omega : \bigcup \mathcal{W}_n \subseteq \bigcup \bigcup_{i > k} \mathcal{G}_i\}$. Since each $\mathcal{W}_n$ is finite and pairwise disjoint, each $A_k$ is a cofinite subset of natural numbers and hence belongs to $\mathcal{F}(\mathcal{I})$(as $\mathcal{I}$ is admissible). Then by $S_1(\mathcal{F}(\mathcal{I}), \mathcal{F}(\mathcal{I}))$, there is a sequence $<k_i :  i \in \omega>$ such that $k_i \in A_i$ for each $i$, and $\{k_1 < k_2 < ... \} \in \mathcal{F}(\mathcal{I})$. Then $\mathcal{W}_{k_1} \subseteq \bigcup_{i > 1} \mathcal{G}_i$. Since $\mathcal{W}_{k_1}$ is finite, $\bigcup \mathcal{W}_{k_1}$ is contained in union of elements of finitely many $\mathcal{G}_i$'s for $i > 1$ and let $\mathcal{F}_1$ denote the collection of all the sets in these finitely many $\mathcal{G}_i$'s. As $\mathcal{W}_{k_2} \subseteq \bigcup_{i > 2} \mathcal{G}_i$. Since $\mathcal{W}_{k_2}$ is finite, $\bigcup \mathcal{W}_{k_2}$ is contained in union of finitely many $\mathcal{G}_i$'s for $i > 2$ and let $\mathcal{F}_2$ denote the collection of all the sets in these finitely many $\mathcal{G}_i$'s. Similarly we get a sequence $<\mathcal{F}_n : n \in \omega>$ such that for each $n$, $\mathcal{F}_n$ is a finite subset of $\mathcal{U}_n$ and $\bigcup \mathcal{W}_{k_n} \subseteq \bigcup \mathcal{F}_n$.   

Since $\mathcal{I}$ is an inverse invariant ideal, for $f : \omega \rightarrow \omega$ such that $f(i) = k_i$, $f(A) \in \mathcal{F}(\mathcal{I})$ implies that $A \in \mathcal{F}(\mathcal{I})$. For each $x \in X$, $\{n \in \omega : x \in \bigcup \mathcal{F}_n\} \supseteq \{n \in \omega : x \in \bigcup \mathcal{W}_{k_n}\}$ and $f(\{n \in \omega : x \in \bigcup \mathcal{W}_{k_n}\}) = \{k_n \in \omega : x \in \bigcup \mathcal{W}_{k_n}\} = \{n \in \omega : x \in \bigcup \mathcal{W}_n\} \cap \{k_n : n \in \omega\} \in \mathcal{F}(\mathcal{I})$, since $\{n \in \omega : x \in \bigcup \mathcal{W}_n\}, \{k_n : n \in \omega\} \in \mathcal{F}(\mathcal{I})$. So $\{n \in \omega : x \in \bigcup \mathcal{F}_n\}  \in \mathcal{F}(\mathcal{I})$. 

Then we have that for each $x \in X$, $\{p \in \omega : x \in \bigcup \mathcal{F}_p\} \in \mathcal{I}$. It follows that $X$ has $\mathcal{I}H$ property.
\end{proof}

In \cite{H42}, Hurewicz property is characterized by Hurewicz measure zero property. Now we characterize $\mathcal{I}$-Hurewicz property by $\mathcal{I}$-Hurewicz measure zero property.

\begin{definition} \cite{H42}
A metric space $(X,d)$ is Hurewicz measure zero if for each sequence $<\epsilon_n: n \in \omega>$ of positive real numbers there is a sequence $<\mathcal{V}_n : n \in \omega>$ such that:
\begin{enumerate}
\item
for each $n$, $\mathcal{V}_n$ is a finite set of open subsets in $X$;
\item
for each $n$, each member of $\mathcal{V}_n$ has d-diameter less than $\epsilon_n$;
\item
$\bigcup_{n \in \omega} \mathcal{V}_n$ is a groupable cover of $X$.
\end{enumerate}
\end{definition}

Now we define its ideal version.

\begin{definition}
A metric space $(X,d)$ is $\mathcal{I}$-Hurewicz measure zero if for each sequence $<\epsilon_n: n \in \omega>$ of positive real numbers there is a sequence $<\mathcal{V}_n : n \in \omega>$ such that:
\begin{enumerate}
\item
for each $n$, $\mathcal{V}_n$ is a finite set of open subsets in $X$;
\item
for each $n$, each member of $\mathcal{V}_n$ has d-diameter less than $\epsilon_n$;
\item
$\bigcup_{n \in \omega} \mathcal{V}_n$ is an $\mathcal{I}$-groupable cover of $X$.
\end{enumerate}
\end{definition}

For the next theorem, we consider inverse invariant ideals to characterize $\mathcal{I}$-Hurewicz property.

\begin{theorem}
Let $\mathcal{I}$ be an inverse invariant ideal such that $S_1(\mathcal{F}(\mathcal{I}),\mathcal{F}(\mathcal{I}))$ holds. If $(X,d)$ is a zero-dimensional separable metric space with no isolated points for which $CDR_{sub}(\mathcal{O},\mathcal{O})$ holds, then following statements are equivalent:
\begin{enumerate}
\item
$X$ has the $\mathcal{I}$-Hurewicz property.
\item
$X$ is $\mathcal{I}$-Hurewicz measure zero with respect to every metric on $X$ which gives $X$ the same topology as $d$ does.
\end{enumerate}
\end{theorem}
\begin{proof}
For $(1) \Rightarrow (2)$, let $X$ has the $\mathcal{I}$-Hurewicz property and let $<\epsilon_n : n \in \omega>$ be a sequence of positive real numbers. For each $n$, define $\mathcal{U}_n = \{U \subset X : U$ is open set with $diam_d(U) < \epsilon_n\}$. Then $\mathcal{U}_n$ is a large open cover of $X$ for each $n$. Since $X$ has $\mathcal{I}$-Hurewicz property, by Corollary \ref{44}, there is a sequence $<\mathcal{V}_n : \in \omega>$ such that for each $n$, $\mathcal{V}_n$ is a finite subset of $\mathcal{U}_n$ and $\bigcup \mathcal{V}_n$ is an $\mathcal{I}_\gamma$ cover of $X$ and so $\mathcal{I}$-groupable cover of $X$. Hence $X$ has $\mathcal{I}$-Hurewicz measure zero with respect to every metric on $X$ which gives $X$ the same topology as $d$ does.

For $(2) \Rightarrow (1)$, let $d$ be an arbitrary metric on $X$ which gives $X$ the same topology as the original one. Let $<\mathcal{U}_n : n \in \omega>$ be a sequence of open covers of $X$. Since $X$ is zero-dimensional metric space, replace $\mathcal{U}_n$ by $\{U \subseteq X : U$ clopen, $diam_d(U) < 1/n$ and $\exists V \in \mathcal{U}_n$ such that $U \subseteq V \}$ for each $n$. Also $X$ is separable metric space, replace last cover by a countable subcover $\{U_m : m \in \omega\}$. Since the cover is countable and sets are clopen, it can be made disjoint clopen cover refining $\mathcal{U}_n$ for each $n$. Also for each $n$, each member of this new cover has $diam_d \leq 1/n$. Also by taking intersections of each new cover with the next new cover we obtain $(3)$. Now name this last cover $\mathcal{U}^\star_n$ for each $n$. So $<\mathcal{U}^\star_n : n \in \omega>$ is a sequence of open covers such that for each $n$:
\begin{enumerate}
\item
$\mathcal{U}^\star_n$ is clopen disjoint cover of $X$ refining $\mathcal{U}_n$;
\item
for each $V \in \mathcal{U}^\star_n$, $diam_d(V) \leq 1/n$;
\item
$\mathcal{U}^\star_{n+1}$ refines $\mathcal{U}^\star_n$.
\end{enumerate} 
Now define a metric $d^\star$ on $X$ by $d^\star(x,y) = 1/(n+1)$ where $n$ is the least such that there exist $U \in \mathcal{U}^\star_n$ with $x \in U$ and $y \notin U$. It can be easily seen that $d^\star$ generates the same topology on $X$ as $d$ does. Since $X$ has $\mathcal{I}$-Hurewicz measure zero with respect to $d^\star$, by setting $\epsilon_n = 1/(n+1)$ for each $n$, there are finite sets $\mathcal{V}_n$ such that $diam_d^\star(U)$ is less than $\epsilon_n( = 1/(n+1))$ whenever $U \in \mathcal{V}_n$, and $\bigcup_{n \in \omega} \mathcal{V}_n$ is $\mathcal{I}$-groupable cover of $X$.

Let $<\mathcal{W}_n : n \in \omega>$ be a sequence of finite subsets of $\bigcup_{n \in \omega} \mathcal{V}_n$ such that $\mathcal{W}_m \cap \mathcal{W}_n = \emptyset$ whenever $m \neq n$, and $\bigcup_{k \in \omega} \mathcal{W}_k = \bigcup_{n \in \omega} \mathcal{V}_n$, and for each $y \in X$, $\{n \in \omega : y \notin \bigcup \mathcal{W}_n\} \in \mathcal{I}$.

Since $\mathcal{W}_1$ is finite, so choose $i_1$ such that $\mathcal{W}_1 \subseteq \bigcup_{i \leq i_1} \mathcal{V}_i$. Then $\bigcup_{i \leq i_1} \mathcal{V}_i$ is finite and exhausted in a finite number of $\mathcal{W}_k$'s and choose $j_1$ such that for each $i \geq j_1$, if $V \in \mathcal{W}_i$, then $V \notin \bigcup_{i \leq i_1} \mathcal{V}_i$.

Now $\mathcal{W}_{j_1}$ is finite, so choose $i_2 > i_1$ such that $\mathcal{W}_{j_1} \subseteq \bigcup_{i_1 < i \leq i_2} \mathcal{V}_i$. Then $\bigcup_{i_1 < i \leq i_2} \mathcal{V}_i$ is finite and exhausted in a finite number of $\mathcal{W}_k$'s and choose $j_2$ such that for each $i \geq j_2$, if $V \in \mathcal{W}_i$, then $V \notin \bigcup_{i_1 < i \leq i_2} \mathcal{V}_i$.

Alternatively, we choose sequences $1 < i_1 < i_2 < ... < i_m < ...$ and $j_0 = 1 < j_1 < j_2 < ...< j_m < ...$ such that :
\begin{enumerate}
\item
Each element of $\mathcal{W}_1$ belongs to $\bigcup_{i \leq i_1} \mathcal{V}_i$;
\item
For each $i \geq j_k$, if $U \in \mathcal{W}_{j_k}$, then $U \notin \bigcup_{i \leq i_k} \mathcal{V}_i$;
\item
Each element of $\mathcal{W}_{j_k}$ belongs to $\bigcup_{i_k < i \leq i_{k+1}} \mathcal{V}_i$.
\end{enumerate} 
Then for each element $V$ of $\mathcal{W}_{j_k}$ has $d^\star$-diameter less than $\epsilon_{i_k} = 1/i_k + 1 \leq 1/k+1$ since $i_k \geq k$. As $V$ is open set in $(X,d^\star)$ and $diam_d^\star(V) < 1/k+1$, then $V \subseteq B_d^\star(x,1/k+1)$, where $B_d^\star(x,1/k+1)$ is an open ball centered at $x \in V$ and of radius $1/k+1$ in $(X,d^\star)$. Now $B_d^\star(x,1/k+1) = \{y \in X : d^\star(x,y) < 1/k+1\}$. So for each $y \in B_d^\star(x,1/k+1), d^\star(x,y) < 1/k+1$, there is $U \in \mathcal{U}^\star_n$ such that $x \in U$ and $y \notin U$ for some $n > k$. So for all $k \leq n$, there is no set $U \in \mathcal{U}^\star_k$ such that $x \in U$ and $y \notin U$. Thus for all $k \leq n$, there is a set $U \in \mathcal{U}^\star_k$ such that $x,y \in U$ for all $x,y \in B_d^\star(x,1/k+1)$, that is, $B_d^\star(x,1/k+1) \subseteq U \in \mathcal{U}^\star_k$ for all $k \leq n$.
Thus, by definition of $d^\star$, each element $V$ of $\mathcal{W}_{j_k}$ is a subset of an element of $\mathcal{U}^\star_k$, each of which in turn is a subset of an element of $\mathcal{U}_k$. For each $k$ and for each element $V$ of $\mathcal{W}_{j_k}$, choose a $U \in \mathcal{U}^\star_k$ with $V \subseteq U$  and let $\mathcal{G}_k$ be the finite set of such chosen $U$'s. So for each $k$, $\bigcup \mathcal{W}_{j_k} \subseteq \bigcup \mathcal{G}_k$.

For each $k$, let $A_k = \{n \in \omega : \bigcup \mathcal{W}_n \subseteq \bigcup \bigcup_{i > k} \mathcal{G}_i\}$. Since each $\mathcal{W}_n$ is finite and pairwise disjoint, each $A_k$ is a cofinite subset of natural numbers and hence belongs to $\mathcal{F}(\mathcal{I})$(as $\mathcal{I}$ is admissible). Then by $S_1(\mathcal{F}(\mathcal{I}), \mathcal{F}(\mathcal{I}))$, there is a sequence $<k_i :  i \in \omega>$ such that $k_i \in A_i$ for each $i$, and $\{k_1 < k_2 < ... \} \in \mathcal{F}(\mathcal{I})$. Then $\mathcal{W}_{k_1} \subseteq \bigcup_{i > 1} \mathcal{G}_i$. Since $\mathcal{W}_{k_1}$ is finite, $\bigcup \mathcal{W}_{k_1}$ is contained in union of elements of finitely many $\mathcal{G}_i$'s for $i > 1$ and for each element of these finitely many $\mathcal{G}_i$'s, there is $W \in \mathcal{U}^\star_1$(by (iii)), each of which in turn is a subset of an element of $\mathcal{U}_1$ and the collection is denoted by $\mathcal{F}_1$. As $\mathcal{W}_{k_2} \subseteq \bigcup_{i > 2} \mathcal{G}_i$. Since $\mathcal{W}_{k_2}$ is finite, $\bigcup \mathcal{W}_{k_2}$ is contained in union of elements of finitely many $\mathcal{G}_i$'s for $i > 2$ and and for each element of these finitely many $\mathcal{G}_i$'s, there is $W \in \mathcal{U}^\star_2$(by (iii)), each of which in turn is a subset of an element of $\mathcal{U}_2$ and the collection is denoted by $\mathcal{F}_2$. Similarly we get a sequence $<\mathcal{F}_n : n \in \omega>$ such that for each $n$, $\mathcal{F}_n$ is a finite subset of $\mathcal{U}_n$ and $\bigcup \mathcal{W}_{k_n} \subseteq \bigcup \mathcal{F}_n$.   

Since $\mathcal{I}$ is an inverse invariant ideal, for $f : \omega \rightarrow \omega$ such that $f(i) = k_i$, $f(A) \in \mathcal{F}(\mathcal{I})$ implies that $A \in \mathcal{F}(\mathcal{I})$. For each $x \in X$, $\{n \in \omega : x \in \bigcup \mathcal{F}_n\} \supseteq \{n \in \omega : x \in \bigcup \mathcal{W}_{k_n}\}$ and $f(\{n \in \omega : x \in \bigcup \mathcal{W}_{k_n}\}) = \{k_n \in \omega : x \in \bigcup \mathcal{W}_{k_n}\} = \{n \in \omega : x \in \bigcup \mathcal{W}_n\} \cap \{k_n : n \in \omega\} \in \mathcal{F}(\mathcal{I})$, since $\{n \in \omega : x \in \bigcup \mathcal{W}_n\}, \{k_n : n \in \omega\} \in \mathcal{F}(\mathcal{I})$. So $\{n \in \omega : x \in \bigcup \mathcal{F}_n\}  \in \mathcal{F}(\mathcal{I})$. 

Then we have that for each $x \in X$, $\{p \in \omega : x \notin \bigcup \mathcal{F}_p\} \in \mathcal{I} $. It follows that $X$ has $\mathcal{I}H$ property.
\end{proof}

\section{Properties of $SS\mathcal{I}H$ spaces} \label{3}

\textbf{Selection principles and $SS\mathcal{I}H$ property}

In \cite{H13}, strongly star-Hurewicz property is characterized by $SS^{\star}_{fin}(\mathcal{O}, \mathcal{O}^{gp})$. Now we present its ideal version to characterize $S\mathcal{I}H$ property for inverse invariant ideals.

\begin{theorem}
Let $\mathcal{I}$ be an inverse invariant ideal such that $S_1(\mathcal{F}(\mathcal{I}), \mathcal{F}(\mathcal{I}))$ holds. Then a space $X$ has $SS\mathcal{I}H$ property if and only if $SS^{\star}_{fin}(\mathcal{O}, \mathcal{I}-\mathcal{O}^{gp})$ holds. 
\end{theorem}
\begin{proof}
Let $X$ be a $SS\mathcal{I}H$ space and $<\mathcal{U}_n : n \in \omega>$ be a sequence of open covers of $X$. Then by $SS\mathcal{I}H$ property of $X$, there is a sequence $<V_n : n \in \omega>$ such that for each $n$, $\mathcal{V}_n$ is a finite subset of $X$ and for each $x \in X$, $\{n \in \omega: x \notin St(V_n, \mathcal{U}_n)\} \in \mathcal{I}$. Then $\{St(V_n,\mathcal{U}_n) : n \in \omega\}$ is an $\mathcal{I}_\gamma$-cover of $X$ and $\mathcal{I}$-groupable cover of $X$.

Conversely, let $SS^{\star}_{fin}(\mathcal{O}, \mathcal{I}-\mathcal{O}^{gp})$ holds and $<\mathcal{V}_n : n \in \omega>$ be a sequence of open covers of $X$. From covers $\mathcal{V}_n$, $n \in \omega$, we define new covers $\mathcal{W}_n \in \Omega$ in the following way :
\begin{center}
$n =1$ : $\mathcal{W}_1 = \mathcal{V}_1$; \\
$n > 1$ : $\mathcal{W}_n = \mathcal{V}_1 \wedge \mathcal{V}_2 \wedge ... \wedge \mathcal{V}_n$. 
\end{center}
For each element of $\mathcal{W}_n$ choose a representation of the form $V_{1,m_1} \cap V_{2,m_2} \cap ... \cap V_{n,m_n}$ with $m_i \in \omega$.
Apply (2) to the sequence $<\mathcal{W}_n : n \in \omega>$ to find for each $n$, a finite set $A_n \subset \mathcal{W}_n$ such that $\{St(A_n, \mathcal{W}_n) : n \in \omega\} \in \mathcal{I}-\mathcal{O}^{gp}$. Thus we can choose finite, pairwise disjoint sets $\mathcal{H}_n$, $n \in \omega$, such that $\{St(A_n, \mathcal{W}_n) : n \in \omega\} = \bigcup_{n \in \omega} \mathcal{H}_n$, and each $x \in X$, $\{n \in \omega : x \notin \bigcup \mathcal{H}_n\} \in \mathcal{I}$.

Now for $n$, define $B_n = \{k \in \omega : \mathcal{H}_k \subseteq \{St(A_i, \mathcal{W}_i) : i > n\}\}$. Since each $\mathcal{H}_n$ is finite and pairwise disjoint, each $B_n$ is cofinite subset of natural numbers and hence belongs to $\mathcal{F}(\mathcal{I})$(as $\mathcal{I}$ is admissible). As $S_1(\mathcal{F}(\mathcal{I}), \mathcal{F}(\mathcal{I}))$ holds, there is a sequence $<k_i :  i \in \omega>$ such that $k_i \in A_i$ for each $i$, and $\{k_1 < k_2 < ... \} \in \mathcal{F}(\mathcal{I})$. Then $\mathcal{H}_{k_1} \subseteq \{St(A_i, \mathcal{W}_i) : i > 1\}$. Let $G_1$ denote the union of the sets $A_i$ such that $St(A_i,\mathcal{W}_i) \in \mathcal{H}_{k_1}$. Now $\mathcal{H}_{k_2} \subseteq \{St(A_i, \mathcal{W}_i) : i > 2\}$. Let $G_2$ denote the union of the sets $A_i$ such that $St(A_i,\mathcal{W}_i) \in \mathcal{H}_{k_2}$. Continuing in this way, we obtain finite subsets $G_n$ of $X$.

Since $\mathcal{I}$ is an inverse invariant ideal, for $f : \omega \rightarrow \omega$ such that $f(i) = k_i$, $f(A) \in \mathcal{F}(\mathcal{I})$ implies that $A \in \mathcal{F}(\mathcal{I})$. For each $x \in X$, $\{n \in \omega : x \in St(G_n,\mathcal{V}_n\} \supseteq \{n \in \omega : x \in \bigcup \mathcal{H}_{k_n}\}$ and $f(\{n \in \omega : x \in \bigcup \mathcal{H}_{k_n}\}) = \{k_n \in \omega : x \in \bigcup \mathcal{H}_{k_n}\} = \{n \in \omega : x \in \bigcup \mathcal{H}_n\} \cap \{k_n : n \in \omega\} \in \mathcal{F}(\mathcal{I})$, since $\{n \in \omega : x \in \bigcup \mathcal{H}_n\}, \{k_n : n \in \omega\} \in \mathcal{F}(\mathcal{I})$. So $\{n \in \omega : x \in St(G_n,\mathcal{V}_n)\}  \in \mathcal{F}(\mathcal{I})$. 

Then the sequence $<G_n : n \in \omega>$ witnesses for $<\mathcal{V}_n : n \in \omega>$ the $SS\mathcal{I}H$ property of $X$.
\end{proof}

From the above theorem, it is natural to consider the selection principle $SS^\star_1(\mathcal{O},\mathcal{I}-\mathcal{O}^{gp})$ that is naturally related to the $SS\mathcal{I}H$ property. For the relationship between these, we need the following definition.

\begin{definition} 
Let $\mathcal{A}$ and $\mathcal{B}$ be families of subsets of the infinite set $S$. Then $CDRF^\star_{sub}(\mathcal{A},\mathcal{B})$ denotes the statement that for each sequence $<A_n : n \in \omega>$ of elements of $\mathcal{A}$ there is a sequence $<B_n : n \in \omega>$ such that for each $n$, $B_n \subseteq A_n$, for $m \neq n$ and for each finite subset $F$ of $S$, $\{St(x,B_m) : x \in F\} \cap \{St(x,B_n) : x \in F\} = \emptyset$, and each $B_n$ is a member of $\mathcal{B}$.
\end{definition}

\begin{theorem} \label{39}
Let a space $X$ satisfies $CDRF^\star_{sub}(\mathcal{O},\mathcal{O})$ and for each sequence $<\mathcal{U}_n : n \in \omega>$ of open covers of a space $X$ there is a sequence $<A_n : n \in \omega>$ of subsets of $X$ such that for each $n, |A_n| \leq n$ and $\{St(A_n, \mathcal{U}_n) : n \in \omega\}$ is an $\mathcal{I}_\gamma$-cover of $X$. Then $X$ satisfies $SS^\star_1(\mathcal{O},\mathcal{I}-\mathcal{O}^{gp})$.
\end{theorem} 
\begin{proof}
Let $<\mathcal{W}_n : n \in \omega>$ be a sequence of open covers of $X$. Since $X$ satisfies $CDRF^\star_{sub}(\mathcal{O},\mathcal{O})$, there is a sequence $<\mathcal{U}_n : n \in \omega>$ such that for each $n$, $\mathcal{U}_n \subseteq \mathcal{W}_n$, for $m \neq n$ and for each finite subset $F$ of $X$, $\{St(x,\mathcal{W}_m) : x \in F\} \cap \{St(x,\mathcal{W}_n) : x \in F\} = \emptyset$, and each $\mathcal{W}_n$ is an open cover of $X$. For each $n$, define $<\mathcal{V}_n : n \in \omega>$ an open cover of $X$ by putting $\mathcal{V}_n = \bigwedge_{(n-1)n/2 < i \leq n(n+1)/2} \mathcal{U}_i$. By applying hypothesis of the theorem to the sequence $<\mathcal{V}_n : n \in \omega>$, we get a sequence $<A_n : n \in \omega>$ of subsets of $X$ such that for each $n$, $|A_n| \leq n$ and $\{St(A_n, \mathcal{V}_n) : n \in \omega\}$ is an $\mathcal{I}_\gamma$-cover of $X$. Now write $A_n = \{x_i : (n-1)n/2 < i \leq n(n+1)/2\}$. Now to prove that the set $\{St(x_n, \mathcal{W}_n) : n \in \omega\}$ is an open $\mathcal{I}$-groupable cover of $X$. Define a sequence $n_1 < n_2 <...< n_k <...$ of natural numbers by $n_k = k(k-1)/2$. Then, for each $x \in X$, $\{k \in \omega : x \notin \bigcup_{n_k < i \leq n_{k+1}} St(x_i, \mathcal{W}_i)\} \subseteq \{k \in \omega : x \notin St(A_k, \mathcal{V}_k) \} \in \mathcal{I}$.
\end{proof}


\section{Properties of $S\mathcal{I}H$ spaces}

\textbf{Selection principles and $S\mathcal{I}H$ property}

In \cite{H13}, star-Hurewicz property is characterized by $U^{\star}_{fin}(\mathcal{O}, \mathcal{O}^{gp})$. Now we present its ideal version to characterize $S\mathcal{I}H$ property for inverse invariant ideals.

\begin{theorem}
Let $\mathcal{I}$ be an inverse invariant ideal such that $S_1(\mathcal{F}(\mathcal{I}), \mathcal{F}(\mathcal{I}))$ holds. Then a space $X$ has $S\mathcal{I}H$ property if and only if $U^{\star}_{fin}(\mathcal{O}, \mathcal{I}-\mathcal{O}^{gp})$ holds. 
\end{theorem}
\begin{proof}
Let $X$ be a $S\mathcal{I}H$ space and $<\mathcal{U}_n : n \in \omega>$ be a sequence of open covers of $X$. Then by $S\mathcal{I}H$ property of $X$, there is a sequence $<\mathcal{V}_n : n \in \omega>$ such that for each $n$, $\mathcal{V}_n$ is a finite subset of $\mathcal{U}_n$ and for each $x \in X$, $\{n \in \omega: x \notin St(\bigcup \mathcal{V}_n, \mathcal{U}_n)\} \in \mathcal{I}$. Then $\{St(\bigcup \mathcal{V}_n,\mathcal{U}_n) : n \in \omega\}$ is an $\mathcal{I}_\gamma$-cover of $X$ and $\mathcal{I}$-groupable cover of $X$.

Conversely, let $U^{\star}_{fin}(\mathcal{O}, \mathcal{I}-\mathcal{O}^{gp})$ holds and $<\mathcal{V}_n : n \in \omega>$ be a sequence of open covers of $X$. From covers $\mathcal{V}_n$, $n \in \omega$, we define new covers $\mathcal{W}_n \in \Omega$ in the following way :
\begin{center}
$n =1$ : $\mathcal{W}_1 = \mathcal{V}_1$; \\
$n > 1$ : $\mathcal{W}_n = \mathcal{V}_1 \wedge \mathcal{V}_2 \wedge ... \wedge \mathcal{V}_n$. 
\end{center}
For each element of $\mathcal{W}_n$ choose a representation of the form $V_{1,m_1} \cap V_{2,m_2} \cap ... \cap V_{n,m_n}$ with $m_i \in \omega$.
Apply (2) to the sequence $<\mathcal{W}_n : n \in \omega>$ to find for each $n$, a finite set $\mathcal{A}_n \subset \mathcal{W}_n$ such that $\{St(\bigcup \mathcal{A}_n, \mathcal{W}_n) : n \in \omega\} \in \mathcal{I}-\mathcal{O}^{gp}$. Thus we can choose finite, pairwise disjoint sets $\mathcal{H}_n$, $n \in \omega$, such that $\{St(\bigcup \mathcal{A}_n, \mathcal{W}_n) : n \in \omega\} = \bigcup_{n \in \omega} \mathcal{H}_n$, and each $x \in X$, $\{n \in \omega : x \notin \bigcup \mathcal{H}_n\} \in \mathcal{I}$.

Now for $n$, define $A_n = \{k \in \omega : \mathcal{H}_k \subseteq \{St(\bigcup \mathcal{A}_i, \mathcal{W}_i) : i > n\}\}$. Since each $\mathcal{H}_n$ is finite and pairwise disjoint, each $A_n$ is cofinite subset of natural numbers and hence belongs to $\mathcal{F}(\mathcal{I})$(as $\mathcal{I}$ is admissible). As $S_1(\mathcal{F}(\mathcal{I}), \mathcal{F}(\mathcal{I}))$ holds, there is a sequence $<k_i :  i \in \omega>$ such that $k_i \in A_i$ for each $i$, and $\{k_1 < k_2 < ... \} \in \mathcal{F}(\mathcal{I})$. Then $\mathcal{H}_{k_1} \subseteq \{St(\bigcup \mathcal{A}_i, \mathcal{W}_i) : i > 1\}$. Let $\mathcal{G}_1$ denote the set of all $V_{1,p}$, that occurs as a first components in the chosen representations of elements of $\mathcal{A}_i$ such that $St(\bigcup \mathcal{A}_i,\mathcal{W}_i) \in \mathcal{H}_{k_1}$. Now $\mathcal{H}_{k_2} \subseteq \{St(\bigcup \mathcal{A}_i, \mathcal{W}_i) : i > 2\}$. Let $\mathcal{G}_2$ denote the set of all $V_{1,p}$, that occurs as a second components in the chosen representations of elements of $\mathcal{A}_i$ such that $St(\bigcup \mathcal{A}_i,\mathcal{W}_i) \in \mathcal{H}_{k_2}$. Continuing in this way, we obtain finite sets $\mathcal{G}_n \subset \mathcal{V}_n$.

Since $\mathcal{I}$ is an inverse invariant ideal, for $f : \omega \rightarrow \omega$ such that $f(i) = k_i$, $f(A) \in \mathcal{F}(\mathcal{I})$ implies that $A \in \mathcal{F}(\mathcal{I})$. For each $x \in X$, $\{n \in \omega : x \in St(\bigcup \mathcal{G}_n,\mathcal{V}_n\} \supseteq \{n \in \omega : x \in \bigcup \mathcal{H}_{k_n}\}$ and $f(\{n \in \omega : x \in \bigcup \mathcal{H}_{k_n}\}) = \{k_n \in \omega : x \in \bigcup \mathcal{H}_{k_n}\} = \{n \in \omega : x \in \bigcup \mathcal{H}_n\} \cap \{k_n : n \in \omega\} \in \mathcal{F}(\mathcal{I})$, since $\{n \in \omega : x \in \bigcup \mathcal{H}_n\}, \{k_n : n \in \omega\} \in \mathcal{F}(\mathcal{I})$. So $\{n \in \omega : x \in St(\bigcup \mathcal{G}_n,\mathcal{V}_n)\}  \in \mathcal{F}(\mathcal{I})$. 

Then the sequence $<\mathcal{G}_n : n \in \omega>$ witnesses for $<\mathcal{V}_n : n \in \omega>$ the $S\mathcal{I}H$ property of $X$.
\end{proof}

From the above theorem, it is natural to ask :
\begin{question}
Is it true that $S^\star_{fin}(\mathcal{O},\mathcal{I}-\Gamma) = S^\star_{fin}(\mathcal{O},\mathcal{I}-\mathcal{O}^{gp})$ ?
\end{question}

Now consider a space $X$ satisfies the following condition closely related to the $S\mathcal{I}H$ property :

$S\mathcal{I}H_{\leq n}$ : For each sequence $<\mathcal{U}_n : n \in \omega>$ of open covers of $X$, there is a sequence $<\mathcal{V}_n : n \in \omega>$ such that for each $n$, $\mathcal{V}_n$ is a finite subset of $\mathcal{U}_n$ having atmost $n$ elements and $x \in X$, $\{n \in \omega : x \notin St(\bigcup \mathcal{V}_n, \mathcal{U}_n)\} \in \mathcal{I}$.

The answer to this question is given by the next theorem. For it we need the following definition.

\begin{definition} 
Let $\mathcal{A}$ and $\mathcal{B}$ be families of subsets of the infinite set $S$. Then $CDR^\star_{sub}(\mathcal{A},\mathcal{B})$ denotes the statement that for each sequence $<A_n : n \in \omega>$ of elements of $\mathcal{A}$ there is a sequence $<B_n : n \in \omega>$ such that for each $n$, $B_n \subseteq A_n$, for $m \neq n$, $\{St(B,A_m) : B \in B_m\} \cap \{St(B,A_n) : B \in B_n\} = \emptyset$, and each $B_n$ is a member of $\mathcal{B}$.
\end{definition}

\begin{theorem} \label{38}
If a space $X$ satisfies $S\mathcal{I}H_{\leq n}$ and $CDR^\star_{sub}(\mathcal{O},\mathcal{O})$, then $X$ satisfies $S^\star_1(\mathcal{O}, \mathcal{I}-\mathcal{O}^{gp})$.
\end{theorem}
\begin{proof}
Let $<\mathcal{G}_n : n \in \omega>$ be a sequence of open covers of $X$. Since $X$ satisfies $CDR^\star_{sub}(\mathcal{O},\mathcal{O})$, there is a sequence $<\mathcal{U}_n : n \in \omega>$ such that for each $n$, $\mathcal{U}_n \subseteq \mathcal{G}_n$, for $m \neq n$, $\{St(B,\mathcal{G}_m) : B \in \mathcal{U}_m\} \cap \{St(B,\mathcal{G}_n) : B \in \mathcal{U}_n\} = \emptyset$, and each $\mathcal{U}_n$ is an open cover of $X$. For each $n$, define $<\mathcal{V}_n : n \in \omega>$ an open cover of $X$ by putting $\mathcal{V}_n = \bigwedge \{\mathcal{U}_i : (n-1)n/2 < i \leq n(n+1)/2\}$. By applying $S\mathcal{I}H_{\leq n}$ to the sequence $<\mathcal{V}_n : n \in \omega>$, we get a sequence $<\mathcal{W}_n : n \in \omega>$ such that for each $n$, $|\mathcal{W}_n| \leq n, \mathcal{W}_n \subseteq \mathcal{V}_n$ and $\{St(\bigcup \mathcal{W}_n, \mathcal{V}_n) : n \in \omega\}$ is an $\mathcal{I}_\gamma$-cover of $X$. Now write $\mathcal{W}_n = \{W_i : (n-1)n/2 < i \leq n(n+1)/2\}$. For each $W_i$ take also the set $U_i \in \mathcal{U}_i$ which is a term in the representation of $W_i$ given above. Now to prove that the set $\{St(U_n, \mathcal{G}_n) : n \in \omega\}$ is an open $\mathcal{I}$-groupable cover of $X$. Define a sequence $n_1 < n_2 <...< n_k <...$ of natural numbers by $n_k = k(k-1)/2$. Then, for each $x \in X$, $\{ k \in \omega : x \notin \bigcup_{n_k < i \leq n_{k+1}} St(W_i, \mathcal{G}_i)\} \subseteq \{k \in \omega : x \notin St(\bigcup \mathcal{W}_k, \mathcal{V}_k)\} \in \mathcal{I}$.
\end{proof}


Now recall that a space $X$ is zero-dimensional if it has a base consisting of clopen sets. 

It is known that in a paracompact Hausdorff space, $\mathcal{I}H$ property and $S\mathcal{I}H$ property are equivalent. By taking stronger conditions than paracompact we drop Hausdorffness.

\begin{theorem}
In a zero-dimensional Lindel$\ddot{o}$f space $X$, $X$ has the $\mathcal{I}H$ property if and only if it has the $S\mathcal{I}H$ property in $X$.
\end{theorem}
\begin{proof}
Suppose that $X$ is $\mathcal{I}H$ space. Let $<\mathcal{U}_n : n \in \omega>$ be a sequence of open covers of $X$. Since $X$ is $\mathcal{I}H$ space, there is a sequence $<\mathcal{W}_n : n \in \omega>$ such that for each $n$, $\mathcal{W}_n$ is a finite subset of $\mathcal{U}_n$ and for each $y \in X$, $\{n \in \omega : y \notin \bigcup \mathcal{W}_n\} \in \mathcal{I}$. Then by $\bigcup \mathcal{W}_n \subseteq St(\bigcup \mathcal{W}_n, \mathcal{U}_n)$, $X$ is $S\mathcal{I}H$ space. 

Now suppose that $X$ has $S\mathcal{I}H$ property. Let $<\mathcal{U}_n : n \in \omega>$ be a sequence of open covers of $X$. Since $X$ is zero dimensional, there is basis $\mathcal{B}$ consisting of clopen sets. Then replace each $U \in \mathcal{U}_n$ by $B \in \mathcal{B}$ such that $B \subseteq U$. Since $X$ is Lindel$\ddot{o}$f space, $\mathcal{U}_n$ is countable clopen cover for each $n$. Now enumerate $\mathcal{U}_n = \{B^n_1,B^n_2,...\}$ for each $n$. Consider $\mathcal{U}^{'}_n = \{B^n_1,B^n_2 \setminus B^n_1,...\}$ for each $n$. Then $\mathcal{U}^{'}_n$ is disjoint clopen cover of $X$ for each $n$. Now apply $S\mathcal{I}H$ property of $X$ to $<\mathcal{U}^{'}_n : n \in \omega>$, there is a sequence $<\mathcal{V}_n : n \in \omega>$ such that for each $n$, $\mathcal{V}_n$ is a finite subset of $\mathcal{U}^{'}_n$ and for each $x \in X$, $\{n \in \omega : x \notin St(\bigcup \mathcal{V}_n,\mathcal{U}^{'}_n)\} \in \mathcal{I}$. Since $\mathcal{U}^{'}_n$ is disjoint, then $St(\bigcup \mathcal{V}_n,\mathcal{U}^{'}_n) = \bigcup \mathcal{V}_n$. Then $\{k \in \omega : x \notin \bigcup \mathcal{V}_k\} \in \mathcal{I}$. For each $n$ and each $V \in \mathcal{V}_n$ there is $U \in \mathcal{U}_n$ such that $V \subseteq U$. Let $\mathcal{G}_n$ be collection of these sets for each $n$. Then $\mathcal{G}_n$ is finite subset of $\mathcal{U}_n$ for each $n$ and $St(\bigcup \mathcal{V}_n,\mathcal{U}^{'}_n) = \bigcup \mathcal{V}_n \subseteq \bigcup \mathcal{G}_n$. Hence $<\mathcal{G}_n : n \in \omega>$ is the required sequence.
\end{proof}

\begin{corollary}
In a zero-dimensional Lindel$\ddot{o}$f space $X$, $X$ has the star-Hurewicz property if and only if it has the Hurewicz property.
\end{corollary}

\end{document}